\documentclass{amsart}


\usepackage{amsmath,amssymb, amstext,amsthm,amsfonts}
\usepackage[dvips]{graphics}

\thanks{Partially supported by CNPq, FAPERJ and PRONEX/DS from Brazil.}

\newcommand{\cl}{\operatorname{Cl}}

\newcommand{\diff}{\operatorname{\mathfrak{X}^1(M)}}

\newcommand{\per}{\operatorname{Per}}
\newcommand{\sink}{\operatorname{Sink}}

\newcommand{\sad}{\operatorname{Saddle}}

\newcommand{\di}{\operatorname{div}}

\newcommand{\dis}{\operatorname{Dis}}

\newcommand{\al}{\alpha}

\newcommand{\be}{\beta}

\newcommand{\ga}{\gamma}

\newcommand{\de}{\delta}

\newcommand{\la}{\lambda}

\newcommand{\si}{\sigma}

\newcommand{\Mundo}{\mathfrak{X}^{1}(M)}

\newcommand{\cum}{C^1}
\newcommand{\ov}{\overline}

\newcommand{\SK}{{\mathcal K}}

\newcommand{\SU}{{\mathcal U}}
\newcommand{\SV}{{\mathcal V}}

\newcommand{\card}{\operatorname{card}}

\newcommand{\deter}{\operatorname{det}}

\newcommand{\ang}{\operatorname{angle}}

\newcommand{\Sink}{\operatorname{Sink}}

\newcommand{\presed}{\operatorname{PreSaddle}_d(X)}
\newcommand{\sadled}{\operatorname{Saddle}_d}

\newtheorem{theorem}{Theorem}[section] 
\newtheorem{lemma}[theorem]{Lemma}     

\newtheorem{proposition}[theorem]{Proposition}
\newtheorem{defi}[theorem]{Definition}

\newtheorem*{ara}{Araujo's Theorem for nonsingular flows}

\title[On Araujo's Theorem for flows]
 {On Araujo's Theorem for flows} 

\author{A. Arbieto, C. A. Morales, B. Santiago}

\address{Instituto de Matem\'atica, Universidade Federal do Rio de Janeiro, P. O. Box 68530, 21945-970 Rio
de Janeiro, Brazil.}
\email{arbieto@im.ufrj.br}
\email{bruno\_santiago@im.ufrj.br}

\begin{document}
\maketitle


\section{Introduction} 
\label{intro}

\noindent
Araujo proved in his thesis \cite{A}
that a $C^1$ generic surface diffeomorphism has either infinitely many sinks (i.e. attracting periodic orbits) or finitely many hyperbolic attractors
with full Lebesgue measure basin.
The goal of this paper is to extend this result to $C^1$ vector fields on compact connected boundaryless manifolds $M$ of dimension $3$
(three-dimensional flows for short).
More precisely, we shall prove that a $C^1$ generic three-dimensional flow without singularities has either
infinitely many sinks or finitely many hyperbolic attractors with full Lebesgue measure basin.
Notice that this result implies Araujo's \cite{A} by the standard
suspension procedure. We stress that
the only available proofs of Araujo's Theorem \cite{A} are the original one \cite{A} and
third author's dissertation \cite{s} under the guidance of the first author (both in Portuguese).
A proof of a weaker result, but with the full Lebesgue measure condition replaced by openess and denseness, was sketched in the draft note \cite{po}.
Let us state our result in a precise way.

The space of three-dimensional flows equipped with the $C^1$ topology will be denoted by $\diff$.
The flow of $X\in\diff$ is denoted by $X_t$, $t\in\mathbb{R}$.
By a {\em singularity} we mean a point $x$ where $X$ vanishes, i.e., $X(x)=0$.
A subset of $\diff$ is {\em residual} if it is a countable intersection of open and dense subsets.
We say that a {\em $C^1$ generic three-dimensional flow satisfies a certain property P} if
there is a residual subset $\mathcal{R}$ of $\diff$ such that P holds for every element of $\mathcal{R}$.
The closure operation is denoted by $\cl(\cdot)$.

Given $X\in\diff$ we denote by $O_X(x)=\{X_t(x): t\in\mathcal{R}\}$
the orbit of a point $x$. By an orbit of $X$ we mean a set $O=O_X(x)$ for some point $x$.
An orbit is {\em periodic} if it is closed as a subset of $M$. The corresponding points are called
{\em periodic points}.
Clearly $x$ is a periodic point if and only if there is a minimal $t_x>0$ satisfying
$X_{t_x}(x)=x$ (we use the notation $t_{x,X}$ to indicate dependence on $X$).
Clearly if $x$ is periodic, then 
$DX_{t_x}(x):T_xM\to T_xM$ is a linear automorphism having
$1$ as eigenvalue with eigenvector $X(x)$.
The remainder eigenvalues (i.e. not corresponding to $X(x)$) will be referred to
as the {\em eigenvalues of $x$}.
We say that a periodic point $x$ is a {\em sink}
if its eigenvalues are less than one (in modulus).
We say that $X$ {\em has infinitely many sinks} if it has infinitely many sinks corresponding to different orbits of $X$.

Given a point $x$ we define the {\em omega-limit set},
$$
\omega(x)=\left\{y\in M : y=\lim_{t_k\to\infty}X_{t_k}(x)\mbox{ for some integer sequence }t_k\to\infty\right\}.
$$
(when necessary we shall write $\omega_X(x)$ to indicate the dependence on $X$.)
We call $\Lambda\subset M$
{\em invariant} if $X_t(\Lambda)=\Lambda$ for all $t\in\mathbb{R}$;
{\em transitive} if there is $x\in\Lambda$ such that
$\Lambda=\omega(x)$; and {\em non-trivial} if it does not reduces to a single orbit.
The {\em basin} of any subset $\Lambda\subset M$ is defined by
$$
W^s(\Lambda)=\{y\in M : \omega(y)\subset \Lambda\}.
$$
(Sometimes we write $W^s_X(\Lambda)$ to indicate dependence on $X$).
An {\em attractor} is a transitive set $A$ exhibiting a neighborhood $U$ such that
$$
A=\displaystyle\bigcap_{t\geq0}X_t(U).
$$

A compact invariant set $\Lambda$ of $X$ is {\em hyperbolic} if
there are a continuous invariant tangent bundle decomposition
$T_\Lambda M=E_\Lambda^s\oplus E^X_\Lambda\oplus E_\Lambda^u$ over $\Lambda$
and positive numbers $K,\lambda$ such that $E^X_x$ is generated by $X(x)$,
$$
\|DX_t(x)/E_x^s\|\leq Ke^{-\lambda t}
\quad\mbox{ and } \quad
\|DX_{-t}(x)/E^u_{X_t(x)}\|\leq K^{-1}e^{\lambda t},
\quad\forall (x,t)\in \Lambda\times\mathbb{R}^+.
$$
With these definitions we can state our result.

\begin{ara}
A $C^1$ generic three-dimensional flow without singularities has either infinitely many sinks or finitely many hyperbolic
attractors with full Lebesgue measure basin.
\end{ara}

This result is closely related to Corollary 1.1 in \cite{mp} which asserts that a $C^1$ generic three-dimensional flow has
either an attractor or a {\em repeller} (i.e. an attractor for the time-reversed flow). Indeed,
Araujo's Theorem for nonsingular flows implies the existence of at least one hyperbolic attractor, but 
in the nonsingular case only.

The idea of the proof of Araujo's Theorem for nonsingular flows is as follows.
To any three-dimensional flow $X$ we define the {\em dissipative region} as the closure
of the periodic points where the product of the eigenvalues is less than $1$ in modulus.
Two important properties of the dissipative region for $C^1$ generic flows without singularities $X$ will be proved.
The first one in Theorem \ref{move-attractor} is that there is a full Lebesgue measure set $L_X$
of points whose omega-limit set intersect the dissipative region.
The second one in Theorem \ref{peo} is that if $X$ has only a finite number of sinks, then the dissipative region is
hyperbolic, and so, splits into a finite disjoint
collection of homoclinic classes and sinks.
These properties allow us to apply the results in \cite{cmp} concerning the neutrality
of the homoclinic classes. In particular, we conclude for $C^1$ generic three-dimensional flows $X$ that
every point in $L_X$ is contained in the basin of the homoclinic classes and singularities in the dissipative region.
To complete the proof we will show in Theorem \ref{fui} that those homoclinic classes
in the dissipative region attracting a positive Lebesgue measure set of points are
in fact hyperbolic attractors.

\section{Dissipative Region}
\label{sec1.5}

\noindent
{\em Hereafter, we will consider three-dimensional flows without singularities only}.

Let $X$ be a three-dimensional flow.
We say that a periodic point $x$ is {\em dissipative}
if $|\det DX_{t_x}(x)|<1$. Denote by $\per_d(X)$ is the set of dissipative periodic points.
The {\em dissipative region} of $X$ is defined by
$$
\dis(X)=\cl(\per_d(X)).
$$
On the other hand, a {\em saddle} of $X$ is a periodic point having eigenvalues of modulus less and bigger than $1$.
We denote by $\sad(X)$ the set of saddles of $X$.
Define the set of dissipative saddles,
$\sad_d(X)=\per_d(X)\cap \sad(X)$.

It follows easily from these definitions that
$\sad_d(X)\cup\sink(X)\subset\per_d(X)$ hence
$\cl(\sad_d(X))\cup \cl(\sink(f))\subset\dis(X)$.
If, additionally, every periodic point is {\em hyperbolic}
(i.e. without eigenvalues of modulus $1$), then
\begin{equation}
\label{split2}
\dis(X)=\cl(\sad_d(X))\cup\cl(\sink(X)).
\end{equation}
In particular, this equality is true for {\em Kupka-Smale flows},
i.e., flows for which every periodic orbit is hyperbolic and the stable and unstable manifolds are in general position \cite{hk}.

Through any saddle $x$ it passes a pair of invariant manifolds, the so-called strong stable and unstable manifolds $W^{ss}(x)$ and $W^{uu}(x)$,
tangent at $x$ to the eigenspaces corresponding to the eigenvalue of modulus less and bigger than $1$ respectively \cite{hps}.
Saturating these manifolds with the flow we obtain the stable and unstable manifolds $W^s(x)$ and $W^u(x)$ respectively.
A {\em homoclinic point} associated to $x$ is a point where these last manifolds meet whereas a homoclinic point $q$ is {\em transverse}
if $T_qM=T_qW^s(x)+T_qW^u(x)$ and $T_qW^s(x)\cap T_qW^u(x)$ is the one-dimensional space generated by $X(q)$.
The {\em homoclinic class} associated to $x$ is the closure of the set of transverse homoclinic points $q$ associated to $x$.
A homoclinic class of $X$ is the homoclinic class associated to some saddle of $X$.
A {\em dissipative homoclinic class} is a homoclinic class associated to a dissipative saddle.

It follows easily from the Birkhoff-Smale Theorem \cite{hk}
that every homoclinic class associated to a dissipative saddle is contained in the dissipative region.
Furthermore, if $\dis(X)$ is hyperbolic, then there is a finite disjoint union
\begin{equation}
\label{split1}
\dis(X)=\left(
\displaystyle\bigcup_{i=1}^rH_i
\right)
\cup
\left(
\displaystyle\bigcup_{j=1}^l s_j
\right),
\end{equation}
where each $H_i$ is a dissipative homoclinic class and each
$s_j$ is the orbit of a sink.

Next we will introduce some results from \cite{cmp}.

Let $\Lambda$ be a compact invariant set of $X$.
We say that $\Lambda$ is {\em Lyapunov stable for $X$}
if for every neighborhood $U$ of $\Lambda$ there is a neighborhood $V\subset U$ of $\Lambda$ such that
$X_t(V)\subset U$, for all $t\geq 0$.
We say that $\Lambda$ is {\em neutral} if $\Lambda=\Lambda^+\cap \Lambda^-$
where $\Lambda^\pm$ is a Lyapunov stable set for $\pm X$.
The following can be proved as in Lemma 2.2 of \cite{cmp}.

\begin{lemma}
\label{neutral}
Let $\Lambda$ be a neutral subset of a three-dimensional flow $X$. If $x\in M$ satisfies
$\omega(x)\cap \Lambda\neq\emptyset$, then $\omega(x)\subset \Lambda$.
\end{lemma}

\begin{proof}
Since $\omega(x)\cap \Lambda\neq\emptyset$, we get $\omega(x)\cap \Lambda^+\neq\emptyset$ and
$\omega(x)\cap\Lambda^-\neq\emptyset$.
Since $\omega(x)\cap\Lambda^+\neq\emptyset$ and $\Lambda^+$ is Lyapunov stable for $x$, we get $\omega(x)\subset \Lambda^+$.
Since $\omega(x)\cap \Lambda^-\neq\emptyset$ and $\Lambda^-$ is Lyapunov stable for $-X$, we get $x\in \Lambda^-$ so
$\omega(x)\subset \Lambda^-$. So $\omega(x)\subset \Lambda^+\cap\Lambda^-=\Lambda$.
\end{proof}

For every subset $\Lambda\subset M$ we define
the {\em weak basin} by
$$
W^s_w(\Lambda)=\{x\in M:\omega(x)\cap \Lambda\neq\emptyset\}.
$$
This is also called
{\em weak region of attraction} \cite{bs}.

\begin{lemma}
\label{homoneutral}
There is a residual subset $\mathcal{R}_2$ of three-dimensional flows $X$
such that if $\dis(X)$ is hyperbolic, then there is a finite disjoint collection of dissipative homoclinic classes $H_1,\cdots, H_r$
and orbits of sinks $s_1,\cdots, s_l$ such that
$$
W^s_w(\dis(X))=\left(\displaystyle\bigcup_{i=1}^rW^s(H_i)\right)\cup\left(\displaystyle\bigcup_{j=1}^lW^s(s_j)\right).
$$
\end{lemma}

\begin{proof}
It follows from the results in Section 3 of \cite{cmp}
that there is a residual subset $\mathcal{R}_2$ of three-dimensional flows
$X$ whose homoclinic classes are all neutral.

Now suppose that $X\in\mathcal{R}_2$ and that $\dis(X)$ is hyperbolic.
Then, we obtain a finite disjoint collection of dissipative homoclinic classes $H_1,\cdots, H_r$
and sinks $s_1,\cdots, s_l$ satisfying (\ref{split1}).
Therefore, if $x\in W^s_w(\dis(X))$, then $\omega(x)$ intersects either $H_i$ or $s_j$
for some $1\leq i\leq r$ and $1\leq j\leq l$.
In the second case we have $x\in W^s(s_j)$ while, in the first, we also
have $x\in W^s(H_i)$ by Lemma \ref{neutral} since every homoclinic class is neutral.
This proves
$$
W^s_w(\dis(X))\subset\left(\displaystyle\bigcup_{i=1}^rW^s(H_i)\right)\cup\left(\displaystyle\bigcup_{j=1}^lW^s(s_j)\right).
$$
As the reversed inclusion is obvious, we are done.
\end{proof}

\section{Weak basin of the dissipative region}

\noindent
Hereafter we denote by $m(\cdot)$ the (normalized) Lebesgue measure of $M$.
The result of this section is the following.

\begin{theorem}
\label{move-attractor}
There is a residual subset $\mathcal{R}_6$ of three-dimensional flows
$X$ for which $m(W^s_w(\dis(X)))=1$
\end{theorem}

To prove it we will need some preliminary notation and results.

Let $\delta_p$ be the Dirac measure supported on a point $p$.
Given a three-dimensional flow $X$ and $t>0$ we define the Borel probability measure
$$
\mu_{p,t}=\frac{1}{t}\int_0^t\delta_{X_s(p)}ds.
$$
(Notation $\mu_{p,t}^X$ indicates dependence on $X$.)

Denote by $\mathcal{M}(p,X)$ as the set of Borel probability measures
$\mu=\lim_{k\to\infty}\mu_{p,t_k}$ for some sequence $t_k\to\infty$.
Notice that each $\mu\in\mathcal{M}(p,X)$ is {\em invariant}, i.e., $\mu\circ X_{-t}=\mu$ for every $t\geq 0$.
With these notations we have the following lemma.

\begin{lemma}
\label{ara1}
For every three-dimensional flow $X$ there is a full Lebesgue measure set $L_X$ of points $x$
satisfying
$$
\int\di Xd\mu\leq0,
\quad\quad\forall \mu\in \mathcal{M}(x,X).
$$
\end{lemma}

\begin{proof}
For every $\delta>0$ we define
$$
\Lambda_\delta(X)=\{x: \exists N_x\in\mathbb{N}\mbox{ such that }|\det DX_t(x)|<(1+\delta)^t, \forall t\geq N_x\}.
$$

We assert that $m(\Lambda_\delta(X))=1$ for every $\delta>0$.
This assertion is similar to one for surface diffeomorphisms given by Araujo \cite{A}.

To prove it we
define
$$
\Lambda_\rho(s)=\{x: \exists N_x\in\mathbb{N}\mbox{ such that }|\det DX_{ns}(x)|<(1+\rho)^{ns}, \forall n\geq N_x\},
\forall s,\rho>0.
$$
We claim that
\begin{equation}
\label{elpe}
m(\Lambda_\rho(s))=1, \quad\quad\forall s,\rho>0.
\end{equation}
Indeed,
take $\epsilon>0$ and for each integer $n$ we define
$$
\Omega(n)=\{x:|\det DX_{ns}(x)|\geq(1+\rho)^{ns}\}.
$$
On the one hand, we get easily that
$$
\Lambda_\rho(s)=\displaystyle\bigcup_{N\in\mathbb{N}}\left(\displaystyle\bigcup_{n\geq N}\Omega(n)\right)^c,
$$
where $(\cdot)^c$ above denotes the complement operation.
On the other hand,
$$
1=\int |\det DX_{ns}(x)|dm\geq \int_{\Omega(n)}|\det DX_{ns}(x)|dm\geq (1+\rho)^{ns}m(\Omega(n)),
$$
yielding $m(\Omega(n))\leq\frac{1}{(1+\rho)^{ns}}$, for all $n$.

Take $N$ large so that
$$
\displaystyle\sum_{n=N}^\infty \frac{1}{(1+\rho)^{ns}}<\epsilon.
$$
Therefore,
$$
m(\Lambda_\rho(s))\geq 1-m\left(\displaystyle\bigcup_{n\geq N}\Omega(n)\right)
\geq 1-\displaystyle\sum_{n=N}^\infty \frac{1}{(1+\rho)^{ns}}>
1-\epsilon.
$$
As $\epsilon>0$ is arbitrary we get (\ref{elpe}). This proves the claim.

Now, we continue with the proof of the assertion.

Fix $0<\rho<\delta$ and $\eta>0$ such that
$$
(1+\eta)(1+\rho)^t<(1+\delta)^t,
\quad\quad\forall t\geq1.
$$
Choose $0<s<1$ satisfying
$$
|\det DX_r(y)-1|\leq \eta,
\quad\quad\forall |r|\leq s,\forall y\in M.
$$
Take $x\in \Lambda_\rho(s)$.
Then, there is an integer $N_x>1$ such that
$$
|\det DX_{ns}(x)|<(1+\rho)^{ns},
\quad\quad\forall n\geq N_x.
$$
Now, if $t\geq N_x$ there are $n\geq N_x$ and $0\leq r<s$ such that
$$
ns\leq t<ns+r.
$$
Thus,
$$
|\det DX_t(x)|=|\det DX_{t-ns}(X_{ns}(x))|\cdot |\det DX_{ns}(x)|
< (1+\eta)(1+\rho)^{ns}.
$$
Then, the choice of $\eta,\rho$ above yields
$
|\det DX_t(x)|<(1+\delta)^t
$ for all $t\geq N_x$
proving
$$
\Lambda_\rho(s)\subset \Lambda_\delta(X).
$$
But (\ref{elpe}) implies $m(\Lambda_\rho(s))=1$ so $m(\Lambda_\delta(X))=1$ proving the assertion.

\vspace{5pt}

To continue with the proof of the lemma, we notice that $\Lambda_{\delta'}(X)\subset \Lambda_\delta(X)$ whenever $\delta'\leq \delta$.
It then follows from the assertion that $L_X$ has full Lebesgue measure, where
$$
L_X=\displaystyle\bigcap_{k\in\mathbb{N}^+}\Lambda_{\frac{1}{k}}(X).
$$
Now, take $x\in L_X$, $\mu\in \mathcal{M}(x,X)$ and $\epsilon>0$.
Fix $k>0$ with $\log\left(1+\frac{1}{k}\right)<\epsilon$.

By definition we have $x\in \Lambda_{\frac{1}{k}}(X)$ and so there is $N_x\in\mathbb{N}^+$ such that
$$
|\det DX_t(x)|^{\frac{1}{t}}<1+\frac{1}{k},
\quad\quad\forall t\geq N_x.
$$
Take a sequence $\mu_{x,t_i}\to \mu$ with $t_i\to\infty$.
Then, we can assume $t_i\geq N_x$ for all $i$. Applying Liouville's Formula \cite{man} we obtain
$$
\int\di Xd\mu=\lim_{i\to\infty}\int\di Xd\mu_{x,t_i}=
\lim_{i\to\infty}\frac{1}{t_i}\int_0^{t_i}\di X(X_s(x))ds=
$$
$$
\lim_{i\to\infty}\frac{1}{t_i}\log|\det DX_{t_i}(x)|
\leq \log\left(1+\frac{1}{k}\right)<\epsilon.
$$
Since $\epsilon>0$ is arbitrary, we obtain the result.
\end{proof}

Again we let $X$ be a three-dimensional flow.
Given $x\in M$ we define
$N_x$ as the orthogonal complement of $X(x)$ in $T_xM$
(when necessary we will write $N^X_x$ to indicate dependence on $X$).
The union $N=\bigcup_{x\in M}N_x$ turns out to be a vector bundle with fiber $N_x$.
Denote by $\pi_x^X: T_xM \to N_x$ the corresponding orthogonal projection.
The {\em Linear Poincar\'e flow} $P^X_t(x): N_x\to N_{X_t(x)}$ of $X$ is defined to be
$$
P^X_t(x)=\pi_{X_t(x)}^X\circ DX_t(x),
\quad\quad\forall (x,t)\in M\times \mathbb{R}.
$$
We shall use the following version of the classical Franks's Lemma \cite{F} (c.f. Appendix A in \cite{bgv})

\begin{lemma}[Franks's Lemma for flows]
\label{frank}
For every three-dimensional flow $X$ and every
neighborhood $W(X)$ of $X$ there is a neighborhood $W_0(X)\subset W(X)$ of $X$ such that
for any $T>0$ there exists $\epsilon>0$ such that for any $Z\in W_0(X)$ and $p\in\per(Z)$,
any tubular neighborhood $U$ of $O_Z(p)$, any partition $0=t_0<t_1<...<t_n=t_{p,Z}$, with $t_{i+1}-t_i<T$
and any family of linear maps $L_i:N_{Z_{t_i}(p)}\to N_{Z_{t_{i+1}}(p)}$ satisfying
$$
\left\|L_i-P^Z_{t_{i+1}-t_i}(Z_{t_i}(p))\right\|<\epsilon,
\quad\quad\forall 0\leq i\leq n-1,
$$
there exists $Y\in W(X)$ with $Y=Z$ along $O_Z(p)$ and outside $U$ such that
$$
P^Y_{t_{i+1}-t_i}(Y_{t_i}(p))=L_i,
\quad\quad\forall 0\leq i\leq n-1.
$$
\end{lemma}

\begin{proof}[Proof of Theorem \ref{move-attractor}]
Denote by $2^M_c$ the set of compact subsets of the surface $M$.
Let $S:\diff\to 2^M_c$ be defined by
$$
S(X)=\cl(\sad_d(X))\cup\cl(\sink(X))
$$
It follows easily from the continuous dependence of the eigenvalues
of a periodic point with respect to $X$ that
this map is {\em lower-semicontinuous},
i.e.,
for every $X\in \diff$ and every open set $W$ with $S(X)\cap W\neq\emptyset$
there is a neighborhood $\mathcal{P}$ of $X$ such that $S(Y)\cap W$ for all $Y\in \mathcal{P}$.
From this and well-known properties of lower-semicontinuous maps
\cite{k1}, \cite{k}, we obtain
a residual subset $\mathcal{A}\subset \diff$ where
$S$ is {\em upper-semicontinuous}, i.e.,
for every $X\in \mathcal{A}$ and every compact subset $K$ satisfying $S(X)\cap K=\emptyset$ there is a neighborhood $\mathcal{D}$ of $X$
such that $S(Y)\cap K=\emptyset$ for all $Y\in\mathcal{D}$.

By the Ergodic Closing Lemma for flows (Theorem 3.9 in \cite{w} or Corollary in p.1270 of \cite{MM})
there is another residual subset $\mathcal{B}$ of three-dimensional flows $X$
such that for every ergodic measure $\mu$ of
$X$
there are sequences $Y^k\to X$ and $p_k$ (of periodic points of $Y^k$) such that
$\mu_{p_k,t_{p_k,Y^k}}^{Y^k}\to\mu$.

By the Kupka-Smale Theorem \cite{hk} there is a residual subset of Kupka-Smale three-dimensional flows $\mathcal{KS}$.

Define $\mathcal{R}_6=\mathcal{A}\cap\mathcal{B}\cap\mathcal{KS}$.
Then, $\mathcal{R}_6$ is a residual subset of three-dimensional flows.

To prove the result we only need to prove
$$
L_X\subset W^s_w(\dis(X)),
\quad\quad\forall X\in\mathcal{R}_6,
$$
where $L_X$ is the full Lebesgue measure set in Lemma \ref{ara1}.

By contradiction assume that it is not.
Then, there are $X\in\mathcal{R}_6$ and $x\in L_X$ satisfying $\omega(x)\cap \dis(X)=\emptyset$, i.e.,
$$
\omega(x)\cap\cl(\per_d(X))=\emptyset.
$$
Since $X\in\mathcal{KS}$, we have
$\dis(X)=S(X)$ from (\ref{split2}). Then,
since $S$ is upper-semicontinuous on $X\in \mathcal{A}$, we can arrange
neighborhoods $U$ of $\omega(x)$ and $W(X)$ of $X$ such that
\begin{equation}
\label{contra}
U\cap(\sad_d(Z)\cup\sink(Z))=\emptyset,
\quad\quad\forall Z\in W(X).
\end{equation}

Put $W(X)$ and $T=1$ in the Franks's Lemma for flows
to obtain $\epsilon>0$ and the neighborhood
$W_0(X)\subset W(X)$ of $X$.
Set
$$
C=\sup\{\|P^Z_t(x)\|:(Z,x,t)\in W(X)\times M\times [0,1]\}
$$
and fix $\delta>0$ such that
$$
|1-e^{-\frac{\delta}{2}}|<\frac{\epsilon}{C}.
$$

Since $M$ is compact, we have $\mathcal{M}(x,X)\neq\emptyset$ and so we can fix $\mu\in\mathcal{M}(x,X)$.
Since $x\in L_X$, we have $\int\di Xd\mu\leq 0$ by Lemma \ref{ara1}.
By the Ergodic Decomposition Theorem \cite{man}
we can assume that $\mu$ is ergodic.
Since $X\in\mathcal{B}$, there are sequences $Y^k\to X$ and $p_k$ (of periodic points of $Y^k$) such that
$\mu_{p_k,t_{p_k,Y^k}}^{Y^k}\to\mu$.
It then follows from Liouville's Formula \cite{man} that
$$
0\geq\int \di Xd\mu=\lim_{k\to\infty}\int \di Xd \mu_{p_k,t_{p_k,Y^k}}^{Y^k}=
$$
$$
\lim_{k\to\infty}\frac{1}{t_{p_k,Y^k}}\int_0^{t_{p_k,Y^k}}\di X(X_s(x))ds=
\lim_{k\to\infty}\frac{1}{t_{p_k,Y^k}}|\det P^{Y^k}_{t_{p_k,Y^k}}(p_k)|
$$
yielding
$$
\lim_{k\to\infty}\frac{1}{t_{p_k,Y^k}}|\det P^{Y^k}_{t_{p_k,Y^k}}(p_k)|\leq 0.
$$
Therefore,
since $Y^k\to X$ and $\mu$ is supported on $\omega(x)\subset U$, we can fix $k$ such that
$$
p_k\in U, \quad Y^k\in W_0(X)\quad\mbox{ and }\quad|\det P^{Y^k}_{t_{p_k,Y^k}}(p_k)|<e^{t_{p_k,Y^k}\delta}.
$$

Once we fix this $k$,
write $t_{p_k,Y^k}=n+r$ where $n\in\mathbb{N}^+$ is the integer part of $t_{p_k,Y^k}$ and $0\leq r<1$.
This induces the partition
$0=t_0<t_1<...<t_{n+1}=t_{p_k,Y^k}$ given by $t_i=i$ for $1\leq i\leq n$.
It turns out that $t_{i+1}-t_i=1$ (for $0\leq i\leq n-1$) and
$t_{n+1}-t_n=t_{p_k,Y^k}-n=r$ therefore,
$t_{i+1}-t_i\leq 1$ for $0\leq i\leq n$.

Define the linear maps
$L_i:N_{Y^k_{t_i}(p)}^{Y^k}\to N_{Y^k_{t_{i+1}}(p)}^{Y^k}$ by
$$
L_i=e^{-\frac{\delta}{2}}P^{Y^k}_{t_{i+1}-t_i}(Y^k_{t_i}(p_k)),
\quad\quad\forall 0\leq i\leq n.
$$
A direct computation shows
$$
\left\|L_i-P^{Y^k}_{t_{i+1}-t_i}(Y^k_{t_i}(p_k))\right\|\leq
|1-e^{-\frac{\delta}{2}}|C<\epsilon,
\quad\quad\forall 0\leq i\leq n.
$$
Then, by the Franks's Lemma for flows,
there exists
$Z\in W(X)$ with $Z=Y^k$ along $O_{Y^k}(p_k)$ such that
$$
P^Z_{t_{i+1}-t_i}(Z_{t_i}(p_k))=L_i,
\quad\quad\forall 0\leq i\leq n.
$$
Consequently,
$t_{p_k,Z}=t_{p_k,Y^k}$ and also
$P^Z_{t_{p_k,Z}}(p_k)=e^{-t_{p_k,Y^k}\frac{\delta}{2}}P^{Y^k}_{t_{p_k,Y^k}}(p_k)$ thus
$$|\det P^Z_{t_{p_k,Z}}(p_k)|=
e^{-t_{p_k,Y^k}\delta}
|\det P^{Y^k}_{t_{p_k,Y^k}}(p_k)|< 1.
$$
Up to a small perturbation if necessary we can assume
that $p_k$ has no eigenvalues of modulus $1$.
Then, $p_k\in \sad_d(Z)\cup\sink(Z)$
by the previous inequality which implies
$p_k\in U\cap(\sad_d(Z)\cup\sink(Z))$.
But $Z\in W(X)$ so we obtain a contradiction by (\ref{contra}) and
the result follows.
\end{proof}

We say that $x$ is a {\em dissipative presaddle} of a three-dimensional flow $X$
if there are sequences $Y_k\to X$ and $x_k\in \sad_d(X_k)$ such that
$x_k\to x$. Compare with \cite{w0}.
Denote by $\sad^*_d(X)$ the set of dissipative presaddles of $X$.

We only need the following elementary property of the set of dissipative presaddles whose
proof is a direct consequence of the definition.

\begin{lemma}
 \label{l5}
For every three-dimenional flow $Y$ and every neighborhood $U$ of $\sad_d^*(Y)$ there is a neighborhood $\mathcal{V}_Y$ of $Y$ such that
$\sad_d^*(Z)\subset U$, for every $Z\in \mathcal{V}_Y$.
\end{lemma}

\section{Lebesgue measure of the basin of hyperbolic homoclinic classes}

\noindent
This section is devoted to the proof of the following result.

\begin{theorem}
\label{fui}
There is a residual subset $\mathcal{R}_{11}$ of three-dimensional flows $Y$ 
such that if $\cl(\sad_d(Y))$ is hyperbolic, then the following properties are equivalent for
every dissipative homoclinic $H$ of $Y$:
\begin{enumerate}
\item[(a)] $m(W^s_Y(H))>0$.
\item[(b)] $H$ is an attractor of $Y$.
\end{enumerate}
\end{theorem}

For this we need the lemma below.
Given a homoclinic class $H=H_X(p)$ of a three-dimensional flow $X$ we denote by
$H_Y=H_Y(p_Y)$ the continuation of $H$, where $p_Y$ is the analytic continuation of $p$
for $Y$ close to $X$ (c.f. \cite{pt}).

\begin{lemma}
\label{local}
There is a residual subset $\mathcal{R}_{12}$ of three-dimensional flows $X$
such that for every hyperbolic homoclinic class $H$
there are an open neighborhood $\mathcal{O}_{X,H}$ of $f$ and a residual subset
$\mathcal{R}_{X,H}$ of $\mathcal{O}_{X,H}$ where the following properties are equivalent:
\begin{enumerate}
\item
$m(W^s_Y(H_Y))=0$ for every $Y\in\mathcal{R}_{X,H}$.
\item
$H$ is not an attractor.
\end{enumerate}
\end{lemma}

\begin{proof}
As in Theorem 4 of \cite{a}, there is a residual subset $\mathcal{R}_{12}$ of three-dimensional flows
$X$ such that, for every homoclinic class $H$ of $X$, the map $Y\mapsto H_Y$
varies continuously at $X$.

Now, let $H$ be a hyperbolic homoclinic class of some $X\in\mathcal{R}_{12}$.
Since $H$ is hyperbolic, we have that $H$ has the local product structure.
From this and the flow version of Proposition 8.22 in \cite{sh} we have that $H$ is {\em uniformly locally maximal}, i.e.,
there are a compact neighborhood $U$ of $H$ 
and a neighborhood $\mathcal{O}_{X,H}$ of $X$
such that
\begin{enumerate}
\item[(a)]
$H=\displaystyle\bigcap_{t\in\mathbb{R}}X_t(U).$
\item[(b)]
$H$ is topologically equivalent to $\displaystyle\bigcap_{t\in\mathbb{R}}Y_t(U)$,
$\forall Y\in \mathcal{O}_{X,H}$.
\end{enumerate}

But the map $Y\mapsto H_Y$ varies continuously at $X$.
From this we can assume up to shrinking $\mathcal{O}_{X,H}$ if necessary that
$H_Y\subset U$, and so,
$H_Y\subset \bigcap_{t\in\mathbb{R}}Y_t(U)$,
for every $Y\in\mathcal{O}_{X,H}$.
Now, we use the equivalence in (b) above and the well-known transitivity of the homoclinic classes
to conclude that $\bigcap_{t\in\mathbb{R}}Y_t(U)$
is a transitive set of $Y$.
From this we get easily that $H_Y= \bigcap_{t\in\mathbb{R}}Y_t(U)$.
We conclude that
\begin{enumerate}
\item[(c)]
$H_Y=\displaystyle\bigcap_{t\in\mathbb{R}}Y_t(U)$ is hyperbolic and topologically equivalent to $H$,
$\forall Y\in\mathcal{O}_{X,H}.$
\end{enumerate}

We claim that if $H$ is not an attractor, then there is a residual subset
$\mathcal{L}_{X,H}$ of $\mathcal{O}_{X,H}$ such that
\begin{equation}
\label{prevent}
m(W^s_Y(H_Y))=0,\quad\quad\forall Y\in \mathcal{L}_{X,H}.
\end{equation}
Indeed, define
$$
\Lambda^N_Y=\displaystyle\bigcap_{0\leq t\leq N}Y_{-t}(U),
\quad\quad\forall (Y,N)\in\mathcal{O}_{X,H}\times (\mathbb{N}\cup\{\infty\}),
$$
and
$$
\mathcal{U}^\epsilon=\{Y\in\mathcal{O}_{X,H}:m(\Lambda^\infty_Y)<\epsilon\},
\quad\quad\forall \epsilon>0.
$$

We assert that $\mathcal{U}^\epsilon$ is open and dense in $\mathcal{O}_{X,H}$, $\forall \epsilon>0$.
To prove it we use an argument from \cite{aapp}.

\vspace{5pt}

For the openess, take $\epsilon>0$ and $Y\in \mathcal{U}^\epsilon$.
It follows from the definitions that
there is $N$ large such that
$m(\Lambda_Y^N)<\epsilon$.
Set $\epsilon_1=\epsilon-m(\Lambda_Y^N)$ thus $\epsilon_1>0$.
Choose $\delta>0$ such that
$$
m(B_\delta(\Lambda_Y^N)\setminus \Lambda_Y^N)<\frac{\epsilon_1}{2}.
$$
(where $B_\delta(\cdot)$ denotes the $\delta$-ball operation).
Since $N$ is fixed we can select a neighborhood $\mathcal{U}_Y\subset \mathcal{O}_{X,H}$ of
$Y$ such that
$$
\Lambda_Z^N\subset B_\delta(\Lambda_Y^N),
\quad\quad \forall Z\in \mathcal{U}_Y.
$$
Therefore,
$$
m(\Lambda_Z^\infty)\leq m(\Lambda_Z^N)\leq m(B_\delta(\Lambda_Y^N))\leq
m(\Lambda_Y^N)+\frac{\epsilon_1}{2}
\leq\frac{m(\Lambda_Y^N)+\epsilon}{2}<\epsilon\quad\quad \forall Z\in \mathcal{U}_Y,
$$
yielding the openess of $\mathcal{U}^\epsilon$.

\vspace{5pt}

For the denseness, take
$\mathcal{D}$ as the $C^2$ flows in $\mathcal{O}_{X,H}$.
Clearly $\mathcal{D}$ is dense in $\mathcal{O}_{X,H}$.
Since $H$ is not an attractor and conjugated to $H_Y$, we have that $H_Y$ is not
an attractor too, $\forall Y\in\mathcal{O}_{X,H}$.
In particular, no $Y\in \mathcal{O}_{X,H}$ has an attractor
in $U$.
Apply Corollary 5.7 in \cite{br} we conclude that
$$
m(\Lambda^\infty_Y)=0, \quad\quad \forall Y\in \mathcal{D}.
$$
From this we have $\mathcal{D}\subset \mathcal{U}^\epsilon$, $\forall \epsilon>0$.
As $\mathcal{D}$ is dense in $\mathcal{O}_{Y,H}$, we are done.

\vspace{5pt}

It follows from the assertion that the intersection
$$
\mathcal{L}_{X,H}=\displaystyle\bigcap_{k\in\mathbb{N}^+}\mathcal{U}^{\frac{1}{k}}
$$
is residual in $\mathcal{O}_{X,H}$. Moreover,
$$
m(\Lambda^\infty_Y)=0, \quad\quad\forall Y\in \mathcal{L}_{X,h}.
$$
Since every $Y\in \mathcal{L}_{X,H}$ is $C^1$ we also obtain
$$
m\left(\displaystyle\bigcup_{n=0}^\infty Y_{-n}(\Lambda^\infty_Y)\right)=0, \quad\quad\forall Y\in \mathcal{L}_{X,h}.
$$
But clearly $W^s_Y(H_Y)= \bigcup_{n=0}^\infty Y_{-n}(\Lambda^\infty_Y)$ so (\ref{prevent}) holds
and the claim follows.

Now, we define
$$
\mathcal{R}_{X,H} = \left\{
\begin{array}{rcl}
\mathcal{L}_{X,H},& \mbox{if }  H\mbox{ is not an attractor}\\
\mathcal{O}_{X,H},& \mbox{otherwise}
\end{array}
\right.
$$

Suppose that $m(W^s_Y(H_Y))=0$ for every $Y\in\mathcal{R}_{X,H}$.
If $H$ were an attractor, then $H_Y$ also is by equivalence thus
$m(W^s_Y(H_Y))>0$, $\forall Y\in\mathcal{O}_{X,H}$, yielding a contradiction. Therefore, $H$ cannot be an attractor.

If, conversely, $H$ is not an attractor, then $\mathcal{R}_{X,H}=\mathcal{L}_{X,H}$ and so
$m(W^s_Y(H_Y))=0$ for every $Y\in\mathcal{R}_{X,H}$ by (\ref{prevent}).
This completes the proof.
\end{proof}

\begin{proof}[Proof of Theorem \ref{fui}]
Let  $\mathcal{R}_{12}$ be as in Lemma \ref{local}.
Define the map $S:\diff\to 2^M_c$ by
$S(X)=\cl(\sad_d(X))$.
This map is clearly lower-semicontinuous, and so, upper semicontinuous in a residual subset
$\mathcal{A}$.

Define $\mathcal{R}= \mathcal{R}_{12}\cap \mathcal{A}$.
Clearly $\mathcal{R}$ is a residual subset of three-dimensional flows.

Define
$$
A=\{f\in \mathcal{R}:\cl(\sad_d(X))\mbox{ is not hyperbolic}\}.
$$

Fix $X\in\mathcal{R}\setminus A$.
Then, $\cl(\sad_d(X))$ is hyperbolic and so
there are finitely many disjoint dissipative homoclinic classes $H^1,\cdots, H^{r_X}$
(all hyperbolic) satisfying
$$
\cl(\sad_d(X))=\displaystyle\bigcup_{i=1}^{r_X}H^i.
$$
As $X\in \mathcal{R}\subset\mathcal{R}_{12}$, we can consider for each $1\leq i\leq r_X$ the neighborhood
$\mathcal{O}_{X, H^i}$ of $X$ as well as its residual subset $\mathcal{R}_{X,H^i}$ given by Lemma \ref{local}.

Define,
$$
\mathcal{O}_X=\displaystyle\bigcap_{i=1}^{r_X}\mathcal{O}_{X,H^i}
\quad \mbox{ and }\quad\mathcal{R}_X=\displaystyle\bigcap_{ì=1}^{r_X}\mathcal{R}_{X,H^i}.
$$

Recalling (c) in the proof of Lemma \ref{local}, we obtain for each $1\leq i\leq r_X$ a compact neighborhood $U_{X,H^i}$ of $H^i$ such that
$$
H^i_Y=\displaystyle\bigcap_{t\in\mathcal{R}}Y_t(U_{X,H^i})\quad\mbox{ is hyperbolic and equivalent to }H^i,
\quad\quad\forall Y\in \mathcal{O}_{Y,H^i}.
$$
As $X\in\mathcal{A}$, $S$ is upper semicontinuous at $X$. So, we can further assume that
$$
\cl(\sad_d(Y))\subset \displaystyle\bigcup_{i=1}^{r_X} U_{X,H^i},
\quad\quad\forall Y\in\mathcal{O}_{X}.
$$
This easily implies
\begin{equation}
\label{tolo}
\cl(\sad_d(Y))=\displaystyle\bigcup_{i=1}^{r_X}H^i_Y,
\quad\quad\forall Y\in\mathcal{O}_{X}.
\end{equation}

Define
$$
\mathcal{O}_{12}=\displaystyle\bigcup_{X\in \mathcal{R}\setminus A}\mathcal{O}_X
\quad\mbox{ and }
\quad
\mathcal{R}'_{12}=\displaystyle\bigcup_{X\in \mathcal{R}\setminus A}\mathcal{R}_X.
$$
We have that $\mathcal{O}_{12}$ is open and $\mathcal{R}_{12}$ is residual
in $\mathcal{O}_{12}$.

Finally we define
$$
\mathcal{R}_{11}=A\cup \mathcal{R}'_{12}.
$$
Since $\mathcal{R}$ is a residual subset of three-dimensional flows, we conclude from Proposition 2.6 in \cite{m}
that $\mathcal{R}_{11}$ also is.

Now, take a $Y\in\mathcal{R}_{11}$ such that $\cl(\sad_d(Y))$ is hyperbolic
and let $H$ be a dissipative homoclinic class of $Y$.
Then,
$H\subset \cl(\sad_d(Y))$ from Birkhoff-Smale's Theorem \cite{hk}.
Since $\cl(\sad_d(Y))$ is hyperbolic, we have $Y\notin A$ so $Y\in \mathcal{R}'_{12}$
thus $Y\in\mathcal{R}_X$ for some $X\in\mathcal{R}\setminus A$.
As $\mathcal{R}_X\subset \mathcal{O}_X$, we obtain $Y\in\mathcal{O}_X$ thus (\ref{tolo}) implies
$H=H^i_Y$ for some $1\leq i\leq r_X$.

If $m(W^{s}_Y(H))>0$, then
$m(W^s_Y(H_Y^i))>0$. But $Y\in\mathcal{R}_X$ so
$Y\in \mathcal{R}_{X,H^i}$. As $f\in \mathcal{R}_{12}$, we conclude from Lemma \ref{local}
that $H^i$ is an attractor. But $H^i$ and $H=H^i_Y$ are equivalent (\ref{tolo}), so, $H^i_Y$ is an attractor too and we are done.
\end{proof}

\section{Hyperbolicity of the dissipative presaddle set}

\noindent
In this section we shall prove the following result.
Hereafter $\card(\sink(X))$ denotes the cardinality of the set of {\em different} orbits of a three-dimensional flow
$X$ on $\sink(X)$.

\begin{theorem}
\label{peo}
There is a residual subset of three-dimensional flows $\mathcal{Q}$
such that if $X\in \mathcal{Q}$ and $\card(\sink(X))<\infty$, then $\sad_d^*(X)$ is hyperbolic. 
\end{theorem}

The proof is based on the auxiliary definition below.

\begin{defi}
We say that a three-dimensional flow $X$ has {\em finitely many sinks robustly} if $\card(X)<\infty$ and, moreover,
there is a neighbourhood $\mathcal{U}_X$ of $X$ such that $\card(\Sink(Y))=\card(\Sink(X))$ for every $Y\in \mathcal{U}_X$. We denote by $S(M)$
the set of three-dimensional flows with this property. 
\end{defi}

Recall that a compact invariant set $\Lambda$ {\em has a dominated splitting} if
there exist a continuous tangent bundle decomposition $N_\Lambda=E_\Lambda\oplus F_\Lambda$
and $T>0$ such that
$$
\left\|P^X_{T}(p)/E_p\right\|\left\|P^X_{-T}(X_T(p))/F_{X_T(p)}\right\|\leq\frac{1}{2},
\quad\quad\forall p\in\Lambda.
$$

The proof of the following result is postponed to Section \ref{secap}.

\begin{proposition}
\label{ladilla2}
If $X\in S(M)$, then
$\sad^*_d(X)$ has a dominated splitting. 
\end{proposition}

\begin{proof}[Proof of Theorem \ref{peo}]
Define $\phi: \diff\to 2^M_c$ by $\phi(X)=\cl(\sink(X))$.
This map is clearly lower semicontinuous, and so, upper semicontinuous in a residual subset $\mathcal{C}$ of $\diff$.
Define,
$$
\mathcal{A}=\{X\in\mathcal{C}:X\mbox{ has infinitely many sinks}\}.
$$
Fix $X\in \mathcal{C}\setminus \mathcal{A}$.
Then, $X\in\mathcal{C}$ and $\card(\sink(X))<\infty$.
Since $\phi$ is upper semicontinuous at $X$, we conclude that there is an open neighborhood $\mathcal{O}_X$ of $X$
such that $\card(\sink(Y))=\card(\sink(X))$ for every $Y\in \mathcal{O}_X$, and so, $\mathcal{O}_X\subset S(M)$.

By the Kupka-Smale theorem \cite{hk} we can find a dense subset
$\mathcal{D}_X\subset \mathcal{O}_X$ formed by $C^2$ Kupka-Smale three-dimensional flows.
Furthermore, we can assume that every $Y\in \mathcal{D}_X$
has neither normally contracting nor normally expanding
irrational tori (see \cite{ar} for the corresponding definition).

Let us prove that $\sad_d^*(Y)$ hyperbolic for every $Y\in\mathcal{D}_X$.
Take any $Y\in \mathcal{D}_X$.
Then $Y\in S(M)$, and so,
$\sad_d^*(Y)$ has a dominated splitting by Proposition \ref{ladilla2}.
On the other hand, it is clear from the definition that every periodic point of $Y$ in
$\sad_d^*(Y)$ is saddle. Then, Theorem B in \cite{ar} implies that
$\sad_d^*(Y)$ is the union of a hyperbolic set and normally contracting irrational tori.
Since no $Y\in \mathcal{D}_X$ has such tori, we are done.

We claim that every $Y\in\mathcal{D}_X$ exhibits an open neighborhood $\mathcal{V}_Y\subset \mathcal{O}_X$
such that $\sad_d^*(Z)$ is hyperbolic, $\forall Z\in\mathcal{V}_Y$.
Indeed, fix $Y\in \mathcal{D}_X$. Since $\sad_d^*(Y)$ is hyperbolic we can choose
a neighborhood $U_Y$ of $\sad_d^*(Y)$ and a neighborhood $\mathcal{V}_Y$ of $Y$ such that any compact invariant set of
any $Z\in \mathcal{V}_Y$ is hyperbolic \cite{hk}.
Applying Lemma \ref{l5} we can assume that $\sad_d^*(Z)\subset U_Y$, for every $Z\in \mathcal{V}_Y$,
proving the claim.

Define
$$
\mathcal{O}_X'=\displaystyle\bigcup_{Y\in\mathcal{D}_X}\mathcal{V}_Y.
$$
Then, $\mathcal{O}_X$ is open and dense in $\mathcal{O}_X$.
Define
$$
\mathcal{O}=\displaystyle\bigcup_{X\in\mathcal{C}\setminus \mathcal{A}}\mathcal{O}_X\quad\quad
\mbox{ and }\quad\quad
\mathcal{O}'=\displaystyle\bigcup_{X\in\mathcal{C}\setminus \mathcal{A}}\mathcal{O}_X'.
$$
It turns out that $\mathcal{O}$ is open and
that $\mathcal{A}\cup \mathcal{O}$ is a residual subset of three-dimensional flows.
Since $\mathcal{O}'$ is open and dense in $\mathcal{O}$, we conclude
that $\mathcal{Q}=\mathcal{A}\cup \mathcal{O}'$ is also a residual subset of three-dimensional flows
(see Proposition 2.6 in \cite{m}).

Now, take $Y\in \mathcal{Q}$ with $\card(\sink(Y))<\infty$.
Then, $Y\notin \mathcal{A}$ and so $Y\in\mathcal{O}_X'$ for some $X\in\mathcal{C}\setminus \mathcal{A}$.
From this we conclude that $\sad_d^*(Y)$ is hyperbolic and we are done.
\end{proof}

\section{Proof of Araujo's Theorem for nonsingular flows}

\noindent
Let $\mathcal{R}_2$, $\mathcal{R}_6$,
$\mathcal{R}_{11}$ and $\mathcal{Q}$ be given by Lemma \ref{homoneutral}, Theorem \ref{move-attractor},
Theorem \ref{fui} and Theorem \ref{peo} respectively.
Define,
$$
\mathcal{R}=\mathcal{R}_2\cap\mathcal{R}_6\cap\mathcal{R}_{11}\cap\mathcal{Q}.
$$
Then, $\mathcal{R}$ is a residual subset of three-dimensional flows.

Now, take $X\in \mathcal{R}$ with $\card(\sink(X))<\infty$.
Since $X\in\mathcal{Q}$, we conclude from Theorem \ref{peo} that $\sad_d^*(X)$ is hyperbolic.
But clearly $\cl(\sad_d(X))\subset\sad_d^*(X)$ thus $\cl(\sad_d(X))$ is hyperbolic too.
As $X$ is Kupka-Smale and $\card(\sink(X))<\infty$, we can apply (\ref{split2}) to conclude that
$\dis(X)$ is hyperbolic.
Since $X\in\mathcal{R}_2$, we can apply Lemma \ref{homoneutral} to obtain
a finite disjoint collection of homoclinic classes $H_1,\cdots, H_r$
and sinks $s_1,\cdots, s_l$ satisfying
$$
W^s_w(\dis(X))=\left(\displaystyle\bigcup_{i=1}^rW^s(H_i)\right)\cup\left(\displaystyle\bigcup_{j=1}^lW^s(s_j)\right).
$$
As $X\in \mathcal{R}_6$, we have
$m(W^s_w(\dis(X)))=1$ by Theorem \ref{move-attractor}.
We conclude that
$$
m\left(
\left(\displaystyle\bigcup_{i=1}^rW^s(H_i)\right)\cup\left(\displaystyle\bigcup_{j=1}^lW^s(s_j)\right)
\right)=1.
$$
Let $1\leq i_1\leq \cdots \leq i_d\leq r$ be such that $m(W^s(H_{i_k}))>0$ for every $1\leq k\leq d$.
As the basin of the remainder homoclinic classes in the collection
$H_1,\cdots,H_r$ are negligible, we can remove them from the above inequality yielding
$$
m\left(
\left(\displaystyle\bigcup_{k=1}^dW^s(H_{i_k})\right)\cup\left(\displaystyle\bigcup_{j=1}^lW^s(s_j)\right)
\right)=1.
$$
Since $f\in \mathcal{R}_{11}$, we have from Theorem \ref{fui} that $H_{i_k}$ is a hyperbolic attractor for every $1\leq k\leq d$.
From this we obtain the result.
\qed

\section{Proof of Proposition \ref{ladilla2}}
\label{secap}

\noindent
First we introduce some basic notations.
Given a three-dimensional flow $X$
and $x\in\sad(X)$,
we denote by
$E^{s}_x$ and $E^{u}_x$ the eigenspaces corresponding to the eigenvalues
of modulus less and bigger than $1$ of $x$
respectively.
We also denote
$N^s_x=\pi_x E^s_x$ and $N^u_x=\pi_x E^u_x$.
Notations $E^{s,X}_x$, $E^{u,X}_x$, $N^{s,X}_x$ and $N^{u,X}_x$ will indicate dependence on $X$.

Notice that if $p\in\sad(Y)$ for some three-dimensional flow $Y$,
then
$$
P_{t_{p,Y}}^Y(p)/N^{s,Y}_p=\lambda(p,Y)I\quad\mbox{ and }\quad
P_{t_{p,Y}}^Y(p)/N^{u,Y}_p=\mu(p,Y)I,
$$
where $I$ is the identity whereas $\lambda(p,Y)$ and $\mu(p,Y)$ are the eigenvalues of $p$ satisfying
$$
|\lambda(p,Y)|<1<|\mu(p,Y)|.
$$

Proposition \ref{ladilla2} is clearly reduced to the following one. 

\begin{proposition}
\label{t.dominalpf}
For every $X\in S(M)$ there are a $C^1$ neighborhood $\SV$
and $T>0$ such that
$$
\left\|P^Y_T(p)/N^{s,Y}_p\right\|\left\|P^Y_{-T}(Y_T(p))/N^{u,Y}_{Y_T(p)}\right\|\leq \frac{1}{2},
\quad\quad\forall (p,Y)\in \sad_d(Y)\times \SV.
$$
\end{proposition}

Its proof is based on two lemmas.

\begin{lemma}
\label{forcaunf}
For every $X\in S(M)$ there are a neighborhood $\mathcal{U}_X$ of $X$
and $0<\lambda<1$ such that
$$
|\lambda(p,Y)|<\lambda^{t_{p,Y}},
\quad\quad\forall (p,Y)\in\sad_d(Y)\times \mathcal{U}_X.
$$
\end{lemma}

\begin{proof}
Since $X\in S(M)$, $\card(\sink(X))<\infty$ and we can fix a neighborhood $W(X)$ of $X$
such that
\begin{equation}
\label{elo}
\card(\sink(Z))=\card(\sink(X)),
\quad\quad\forall Z\in W(X).
\end{equation}
Applying the Franks's Lemma for flows with $T=1$ we obtain a neighborhood $W_0(X)\subset W(X)$ and $\epsilon>0$.

We claim that $\mathcal{U}_X=W_0(X)$ satisfies the conclusion of the lemma.
Indeed, suppose by contradiction that this is not true.
Set
$$
C=\sup\{\|P^Z_t(x)\|:(Z,x,t)\in W(X)\times M\times [0,1]\}
$$
and choose $0<\delta<1$ such that
$$
|(1-\delta)^t-1|<\frac{\epsilon}{C}, \quad\quad\forall 0\leq t\leq 1.
$$
Since the conclusion is not true for $\mathcal{U}_X=W_0(X)$ we can arrange $Y\in W_0(X)$ and $p\in\sad_d(Y)$
such that $(1-\delta)^{t_{p,Y}}\leq |\lambda|$
(for simplicity we write $\lambda=\lambda(p,Y)$ and $\mu=\mu(p,Y)$).
Since $p\in\sad_d(Y)$ we also have
$|\lambda\mu|<1$ thus $|\mu|=|\lambda|^{-1}|\lambda\mu|<(1-\delta)^{-t_{p,Y}}$ yielding
$$
|\mu|<(1-\delta)^{-t_{p,Y}}.
$$
Now write $t_{p,Y}=n+r$ where $n$ is the integer part of $t_{p,Y}$ and $0\leq r<1$.
Consider the partition $t_i=i$ for $0\leq i\leq n$ and $t_{n+1}=t_{p,Y}$.
It follows that $t_{i+1}-t_i\leq 1$ for every $0\leq i\leq n$.
Consider a small tubular neighborhood $U$
of $O_Y(p)$, disjoint of $\Sink(Y)$.
Define the maps $L_i:N_{Y_{t_i}(p)}\to N_{Y_{t_{i+1}}(p)}$ by
$$
L_i=(1-\delta)^{t_{i+1}-t_i}P^Y_{t_{i+1}-t_i}(Y_{t_i}(p)),
\quad\quad\forall 0\leq i\leq n.
$$
The choice of $\delta$ implies
$$
\left\|L_i-P^Y_{t_{i+1}-t_i}(Y_{t_i}(p))\right\|\leq |(1-\delta)^{t_{i+1}-t_i}-1|C<\epsilon,
\quad\quad\forall 0\leq i\leq n.
$$
By Franks's Lemma for flows, there exists $Z\in W(X)$ such that $Z=Y$ along $O_Y(p)$ and
outside $U$ satisfying $P^Z_{t_{i+1}-t_i}(Z_{t_i}(p))=L_i$ for every $0\leq i\leq n$. This implies that
$$
P^Z_{t_{p,Y}}(p)=\displaystyle\prod_{i=0}^{n+1}L_i=(1-\delta)^{t_{p,Y}}P^Y_{t_{p,Y}}(p),
$$
and thus the eigenvalues of $P^Z_{t_{p,Y}}(p)$ are $(1-\delta)^{t_{p,Y}}\lambda$ (of modulus less than $1$) and
$(1-\delta)^{t_{p,Y}}\mu$ (of modulus less than $1$ too).
Therefore, $p\in\sink(Z)$.
Since all sinks of $Y$ are located outside $U$, they are also sinks for $Z$.
This implies that $\card(\Sink(Z))>\card(\Sink(Y))=\card(\sink(X))$. Since $Z\in W(X)$ we obtain a contradiction from (\ref{elo}) proving the result.
\end{proof}

The orthogonal complement of a linear subspace $E$ of $\mathbb{R}^2$ is denoted by
$E^\perp$.
The {\em angle} between linear spaces $E,F$ of $\mathbb{R}^2$ is defined by $\ang(E,F)=\|L\|$, where $L: E\to E^\perp$ is the linear operator
satisfying $F=\{u+L(u):u\in E\}$.

\begin{lemma}
\label{angle}
For every $X\in S(M)$ there exist a neighborhood
$\SU$ and $\al>0$ such that
$$
\ang(N^{s,Y}_p,N^{u,Y}_p)>\al,
\quad\quad \forall (p,Y)\in\sad_d(Y)\times\SU.
$$
\end{lemma}

\begin{proof}
Let $\mathcal{U}_X$ and $0<\lambda<1$ be given by Lemma \ref{forcaunf}.
Since $X\in S(M)$ we can also assume that
$\card(\sink(Z))=\card(\sink(X))$ for every $Z\in \mathcal{U}_X$.

Put $W(X)=\mathcal{U}_X$ in the Franks's Lemma for flows with $T=1$
to obtain the neighborhood $W_0(X)\subset W(X)$ of $X$ and $\epsilon>0$.
Set
$$
C=\sup\{\|P^Z_t(x)\|:(Z,x,t)\in W(X)\times M\times [0,1]\}
$$
and choose $\alpha>0$ such that
$$
\left(\frac{2}{1-\lambda}+1\right)\alpha<\frac{\epsilon}{C}.
$$

We claim that $\SU=W_0(X)$ and $\al$ as above
satisfies the conclusion of the lemma.
Indeed, suppose by contradiction that this is not true.
Then, there are $(p,Y)\in\sad_d(Y)\times W_0(X)$
satisfying
$$
\ang(N^{s,Y}_p,N^{u,Y}_p)\leq \alpha.
$$
Clearly we can fix a tubular neighborhood $U$ of $O_Y(p)$
disjoint from $\sink(Y)$.

With respect to the orthogonal splitting
$N^Y_p=N^{s,Y}_p\oplus [N^{s,Y}_p]^{\perp}$ one has the matrix expression
\begin{equation}
\label{matrix1}
P^Y_{t_{p,Y}}(p) =\begin{bmatrix}
 \la(p,Y) & \frac{\mu(p,Y)-\la(p,Y)}{\gamma}\\
 0 & \mu(p,Y)
\end{bmatrix}.
\end{equation}
To simplify we write $\gamma=\ang(N^{s,Y}_p,N^{u,Y}_p)$ so
$$
0<\gamma<\alpha.
$$
Define
\begin{equation}
\label{matrix2}
A=\begin{bmatrix}
 1 & 0\\
 \left(\frac{\lambda(p,Y)+\mu(p,Y)}{\lambda(p,Y)-\mu(p,Y)}\right)\gamma & 1
\end{bmatrix},
\end{equation}
with respect to the splitting $N^Y_p=N^{s,Y}_p\oplus [N^{s,Y}_p]^{\perp}$.

Now write $t_{p,Y}=n+r$ where $n$ is the integer part of $t_{p,Y}$ and $0\leq r<1$.
Consider the partition $t_i=i$ for $0\leq i\leq n$ and $t_{n+1}=t_{p,Y}$.
Clearly $t_{i+1}-t_i\leq 1$ for every $0\leq i\leq n$.

Define the sequence of linear maps $L_i=P^Y_{t_{i+1}-t_i}(Y_{t_i}(p))$ for $0\leq i\leq n$
and $L_{n+1}=A\circ P^Y_r(Y_{t_n}(p))$.
Since $(p,Y)\in\sad_d(Y)\times \mathcal{U}_X$, Lemma \ref{forcaunf} implies
$$
|\lambda(p,Y)-\mu(p,Y)|\geq |\mu(p,Y)|-|\lambda(p,Y)|>1-\lambda^{t_{p,Y}}>1-\lambda.
$$
Therefore,
$$
\left|
\frac{\lambda(p,Y)+\mu(p,Y)}{\lambda(p,Y)-\mu(p,Y)}
\right|
=\left|\frac{2\lambda(p,Y)}{\lambda(p,Y)-\mu(p,Y)}-1\right|
\leq \frac{2}{1-\lambda}+1,
$$
From this we obtain
$$
\left\|L_i-P^Y_{t_{i+1}-t_i}(Y_{t_i}(p))\right\|\leq \|A-I\| C=
\left|\frac{\lambda(p,Y)+\mu(p,Y)}{\lambda(p,Y)-\mu(p,Y)}\right|\gamma C
\leq
$$
$$
\left(\frac{2}{1-\lambda}+1\right)\alpha C<\epsilon,
$$
thus
$$
\left\|L_i-P^Y_{t_{i+1}-t_i}(Y_{t_i}(p))\right\|<\epsilon,
\quad\quad\forall 0\leq i\leq n+1.
$$
By Franks's Lemma for flows, there exists $Z\in W(X)$ such that $Z=Y$ along $O_Y(p)$ and outside $U$
satisfying $P^Z_{t_{i+1}-t_i}(Z_{t_i}(p))=L_i$ for every $0\leq i\leq n$. This implies that
$$
P^Z_{t_{p,Z}}(p)=\displaystyle\prod_{i=0}^{n+1}L_i=A\circ P^Y_{t_{p,Y}}(p),
$$
But now a direct computation using (\ref{matrix1}) and (\ref{matrix2})
implies that $P^Z_{t_{p,Z}}(p)$ is traceless and $\det P^Z_{t_{p,Z}}(p)=\lambda(p,Y)\mu(p,Y)$.
As $p\in\sad_d(Y)$, we have $|\lambda(p,Y)\mu(p,Y)|<1$ thus $P^Z_{t_{p,Y}}(p)$ has a pair of complex eigenvalues of modulus
less than $1$. Therefore, $p\in\sink(Z)$.
Since all sinks of $Y$ are located outside $U$, they are also sinks for $Z$.
This implies that $\card(\Sink(Z))>\card(\Sink(Y))=\card(\sink(X))$, a contradiction which ends the proof.
\end{proof}

\begin{proof}[Proof of Proposition \ref{t.dominalpf}]
It suffices to prove that there exists
$T>0$ such that for every $Y$ close to $X$ and every $p\in\sad_d(Y)$ there exists
$0\leq t\leq T$ such that
$$
\left\|P^Y_t(p)/N^{s,Y}_p \right\|\left\|P^{Y}_{-t}(X_t(p))/N^{u,Y}_p\right\|<\frac{1}{2}.
$$
Otherwise, for every $T>0$ there exists a flow $Y$, as close as we want to $X$, and a periodic point
$p\in\sad_d(Y)$ such that
\begin{equation}
\label{naodomina}
\left\|P^Y_t(p)/N^{s,Y}_p\right\|\left\|P^{Y}_{-t}(X_t(p))/N^{u,Y}_p \right\|\geq\frac{1}{2},
\quad\quad\forall 0\leq t\leq T.
\end{equation} 
We can assume that $t_{p,Y}\to\infty$ as $Y\to X$.
Indeed, by Lemma \ref{forcaunf}
\begin{eqnarray*}
\left\|P^{Y}_{kt_{p,Y}}(p)/N^{s,Y}_p\right\|\left\|P^{Y}_{-kt_{p,Y}}(p)/N^{u,Y}_p\right\|&\leq&\lambda^{kt_{p,Y}}\\
&\to& 0\:\textrm{as}\:k\to\infty,
\end{eqnarray*}
and thus there exists $k_0$, which depends only upon $\lambda$, such that 
$$
\left\|P^{Y}_{k_0t_{p,Y}}(p)/N^{s,Y}_p\right\|\left\|P^{Y}_{-k_0t_{p,Y}}(p)/N^{u,Y}_p\right\|<\frac{1}{2}.
$$ 
This implies $t_{p,Y}\geq\frac{T}{k_0}$, and since $T$ is large, $t_{p,Y}$ is also large.

Let $\mathcal{U}_X$ and $0<\lambda<1$ be given by Lemma \ref{forcaunf}. Let $\mathcal{U}$ and $\alpha$ be given by Lemma \ref{angle}.
Without loss of generality we can assume that $\mathcal{U}=\mathcal{U}_X$.
Put $W(X)=\mathcal{U}$ in the Franks's Lemma for flows with $T=1$ to obtain
the neighborhood $W_0(X)\subset W(X)$ and $\epsilon>0$.
Set
$$
C=\sup\{\|P^Z_t(x)\|:(Z,x,t)\in W(X)\times M\times [0,1]\}.
$$

Choose $\epsilon_0>0$, $\epsilon_1>0$ and $m\in\mathbb{N}^+$ satisfying
\begin{equation}
\label{bala}
(2\epsilon_0+\epsilon_0^2)C\leq\epsilon,\quad
(1+\epsilon_1)\lambda<1,
\quad\quad\epsilon_1<\frac{\alpha}{1+\alpha}\epsilon_0
\end{equation}
and
\begin{equation}
\label{lapa}
\epsilon_1(1+\epsilon_1)^{m}\geq\frac{2}{\alpha}+4.
\end{equation}

Taking $Y$ close to $X$ we can assume that
\begin{equation}
\label{pepa}
Y\in W_0(X),\quad
\tau>2m
\quad\mbox{ and }\quad
\frac{\|P^Y_m(p)/N^{s,Y}_{p}\|}{\|P^Y_m(p)/N^{u,Y}_{p}\|}>\frac{1}{2}
\end{equation}
with $\tau=t_{p,Y}$.
By the last inequality above we can fix unitary vectors $u\in N^{u,Y}_p$ and $v\in N^{s,Y}_p$ satisfying
\begin{equation}
\label{ll}
\frac{\|P^Y_m(p)v\|}{\|P^Y_m(p)u\|}>\frac{1}{2}.
\end{equation}

Define linear maps $P,S: N_p\to N_p$ by
$$
\left\{
\begin{array}{rcl}
P(v)=&v\\
P(u)=&u+\epsilon_1v
\end{array}
\right.
$$
and
$$
\left\{
\begin{array}{rcl}
S(v)=&v\\
S(u)=&u-\bigg(\epsilon_1(1+\epsilon_1)^{\tau-2m-1}\lambda(p,Y)\mu^{-1}(p,Y)\bigg)v.
\end{array}
\right.
$$
Applying Lemma \ref{forcaunf}, the last two inequalities in (\ref{bala}) and Lemma II.10 in \cite{M} we obtain
$$
\|P-I\|\leq \left(\frac{1+\alpha}{\alpha}\right)\epsilon_1<\epsilon_0
$$
and
$$
\|S-I\|\leq \left(\frac{1+\alpha}{\alpha}\right)\bigg(\epsilon_1(1+\epsilon_1)^{\tau-2m-1}\lambda(p,Y)\mu^{-1}(p,Y)\bigg)<
$$
$$
\left(\frac{1+\alpha}{\alpha}\right)\epsilon_1
<\epsilon_0.
$$
Therefore,
\begin{equation}
\label{diablodado}
\|P-I\|<\epsilon_0
\quad\quad\mbox{ and }\quad\quad
\|S-I\|<\epsilon_0.
\end{equation}

Set $\tau=n+r$ where $n$ is the integer part of $\tau$ and $0\leq r<1$.
Define the partition
$t_i=i$ for $0\leq i\leq n$ and $t_{n+1}=\tau$.

Take linear maps $T_j: N_{Y_{t_j}(p)}^Y\to N_{Y_{t_j}(p)}^Y$, $0\leq j\leq n+1$,
$$
T_j/N^{s,Y}_{Y_{t_j}(p)} = \left\{
\begin{array}{rcl}
(1+\epsilon_1)I,& \mbox{if} & 0\leq j\leq m\\
(1+\epsilon_1)^{-1}I, & \mbox{if} & m+1\leq j\leq n+1
\end{array}
\right.
$$
and
$$
T_j/N^{u,Y}_{Y_{t_j}(p)}=I,
\quad\quad\forall 0\leq j\leq n+1,
$$

Applying (\ref{bala}) and Lemma II.10 in \cite{M} we obtain
$$
\|T_j-I\|\leq \frac{1+\alpha}{\alpha}\epsilon_1<\epsilon_0, \quad\quad\forall 0\leq j\leq m,
$$
and
$$
\|T_j-I\|\leq \frac{1+\alpha}{\alpha}\frac{\epsilon_1}{1+\epsilon_1}<\epsilon_0,
\quad\quad \forall m+1\leq j\leq n+1.
$$
Therefore,
\begin{equation}
\label{extra}
\|T_j-I\|\leq \epsilon_0, \quad\quad\forall 0\leq j\leq n+1.
\end{equation}

Define the linear maps $L_j:N_{Y_{t_j}(p)}M\to N_{Y_{t_{j+1}}(p)}M$ by
$$
L_j = \left\{
\begin{array}{rcl}
T_1\circ P^Y_1(p)\circ P,& \mbox{if} & j=0\\
T_{j+1}\circ P^Y_{t_{j+1}-t_j}(Y_{t_j}(p)), & \mbox{if} & 1\leq j\leq n\\
S\circ T_0\circ P^Y_r(Y_{t_n}(p)), & \mbox{if} & j=n+1.
\end{array}
\right.
$$

We can use the first inequality in (\ref{bala}), (\ref{diablodado}) and (\ref{extra}) as in \cite{PS} to prove
$$
\|L_j-P^Y_{t_{j+1}-t_j}(Y_{t_j}(p))\|\leq \epsilon,
\quad\quad\forall 0\leq j\leq n.
$$

Since $Y\in W_0(X)$, we can apply the Franks's Lemma for flows to
obtain a three-dimensional flow $Z$ with
\begin{equation}
\label{laleli}
 Z\in W(X)
\end{equation}
such that
$Z=Y$ along $O_Y(p)$ (thus $p$ is a periodic point of $Z$ with $t_{p,Z}=\tau$)
and $P^Z_{t_{j+1}-t_j}(Z_{t_j}(p))=L_j$ for $0\leq j\leq n+1$.

It follows that
$$
P^Z_{t_{p,Z}}(p)=\displaystyle\prod_{j=0}^{n+1}L_j.
$$
Direct computations together with the definitions of $P$ and $S$ then show
$$
P^Z_{t_{p,Z}}(p)v=(1+\epsilon_1)^{-\tau+2m+1}\lambda(p,Y)v
$$
and
\begin{eqnarray*}
P^Z_{t_{p,Z}}(p)u & = & S\bigg(\mu(p,Y) u-\epsilon_1(1+\epsilon_1)^{\tau-2m-1}\lambda(p,Y) v\bigg) \\
& = & \mu(p,Y)S\bigg(u-\epsilon_1(1+\epsilon_1)^{\tau-2m-1}\lambda(p,Y)\mu^{-1}(p,Y)v\bigg)\\
& = & \mu(p,Y)\bigg(u-
\epsilon_1(1+\epsilon_1)^{\tau-2m-1}\lambda(p,Y)\mu^{-1}(p,Y)v+\\
& &\epsilon_1(1+\epsilon_1)^{\tau-2m-1}\lambda(p,Y)\mu^{-1}(p,Y)v
\bigg)\\
& = & \mu(p,Y)u.
\end{eqnarray*}

As $|(1+\epsilon_1)^{-\tau+2m+1}\lambda(p,Y)|<1$ and $|\mu(p,Y)|>1$ we obtain
that $p\in \sad(Z)$,
\[
N^{s,Z}_p=N^{s,Y}_p,
\lambda(p,Z)=(1+\epsilon_1)^{-\tau+2m+1}\lambda(p,Y),
N^{u,Z}_p=N^{u,Y}_p,
\mu(p,Z)=\mu(p,Y).
\]
In particular,
$$
|\det(DZ_{t_{p,Z}}(p))|=(1+\epsilon_1)^{-\tau+2m+1}|\lambda(p,Y)\mu(p,Y)|\leq
|\det(DY_{t_{p,Y}}(p))|<1
$$
proving
\begin{equation}
 \label{pape}
p\in \sad_d(Z).
\end{equation}

Now, we finish as in p. 971 of \cite{PS}.

Since $v\in N^{s,Z}_p$ and $u\in N^{u,Z}_p$ we obtain,
$u_1=P^Z_m(p)(u)\in N^{u,Z}_{Z_m(p)}$ and $u_2=P^Z_m(p)(v)\in N^{s,Z}_{Z_m(p)}$.
But
$$
u_1=P^Y_m(p)(u)+\epsilon_1(1+\epsilon_1)^mP^Y_m(p)(v)\quad\mbox{ and }\quad
u_2=(1+\epsilon_1)^mP^Y_m(p)(v).
$$
As $\epsilon_1u_2\in N^{s,Z}_{Z_m(p)}$, Lemma II.10 in \cite{M} yields
$$
\|u_1-\epsilon_1u_2\|\geq\frac{\beta}{1+\beta}\|u_1\|,
\quad\mbox{ where }\quad
\beta=\ang(N^{u,Z}_{Z_{m}(p)},N^{s,Z}_{Z_{m}(p)}).
$$

Consequently,
$$
\|P^Y_m(p)(u)\|=\|u_1-\epsilon_1u_2\|\geq\frac{\beta}{1+\beta}\|u_1\|
\geq
$$
$$
\frac{\beta}{1+\beta}\left(\epsilon_1(1+\epsilon_1)^m\|P^Y_m(p)(v)\|-
\|P^Y_m(p)(u)\|\|\right)
\overset{(\ref{ll})}{\geq}
$$
$$
\frac{\beta}{1+\beta}\left(
\frac{\epsilon_1(1+\epsilon_1)^m}{2}-1
\right)\|P^Y_m(p)(u)\|
$$
yielding
$$
\frac{1+\beta}{\beta}\geq\frac{\epsilon_1(1+\epsilon_1)^m}{2}-1.
$$
As $\beta=\ang(N^{u,Z}_{Z_{m}(p)},N^{s,Z}_{Z_{m}(p)})$ we obtain
$$
\ang(N^{u,Z}_{Z_{m}(p)},N^{s,Z}_{Z_{m}(p)})\leq\frac{2}{\epsilon_1(1+\epsilon_1)^m-4}.
$$
Applying (\ref{lapa}) we arrive to $\ang(N^{u,Z}_{Z_{m}(p)},N^{s,Z}_{Z_{m}(p)})\leq\alpha$. But now (\ref{pepa}) and
(\ref{laleli}) contradict Lemma \ref{forcaunf}
and the proof follows.
\end{proof}


\vspace{20pt}




\begin{thebibliography}{10}



\bibitem{a}
Abdenur, F.,
Generic robustness of spectral decompositions,
{\em Ann. Sci. \'Ecole Norm. Sup.} (4) 36 (2003), no. 2, 213--224.







\bibitem{aapp}
Alves, J.F., Ara\'{u}jo, V., Pacifico, M.J., Pinheiro, V.,
On the volume of singular-hyperbolic sets,
{\em Dyn. Syst.} 22 (2007), no. 3, 249–267. 




\bibitem{A}
Araujo, A.,
{\em Exist\^encia de atratores hiperb\'olicos para difeomorfismos de superficies} (Portuguese),
Preprint IMPA S\'erie F, No 23/88, 1988.




\bibitem{ar}
Arroyo, A., Rodriguez Hertz, F.,
Homoclinic bifurcations and uniform hyperbolicity for three-dimensional flows,
{\em Ann. Inst. H. Poincar\'e Anal. Non Lin\'eaire} 20 (2003), no. 5, 805--841.





\bibitem{bs}
Bhatia, N.P., Szeg\"o, G.P.,
{\em Stability theory of dynamical systems},
Die Grundlehren der mathematischen Wissenschaften, Band 161 Springer-Verlag, New York-Berlin 1970.





\bibitem{bgv}
Bonatti, C.,  Gourmelon, N., Vivier, T.,
Perturbations of the derivative along periodic orbits,
{\em Ergodic Theory Dynam. Systems} 26 (2006), no. 5, 1307--1337. 






\bibitem{br}
Bowen, R., Ruelle, D.,
The ergodic theory of Axiom A flows,
{\em Invent. Math.} 29 (1975), no. 3, 181--202. 


\bibitem{cmp}
Carballo, C.M., Morales, C.A., Pacifico, M.J.,
Homoclinic classes for generic $C^1$ vector fields,
{\em Ergodic Theory Dynam. Systems} 23 (2003), no. 2, 403--415. 




\bibitem{cp}
Crovisier, S., Pujals, E.R.,
Essential hyperbolicity and homoclinic bifurcations:
a dichotomy phenomenon/mechanism for diffeomorphisms,
Preprint arXiv:1011.3836v1 [math.DS] 16 Nov 2010.






\bibitem{F}
Franks, F.,
Necessary conditions for stability of diffeomorphisms,
{\em Trans. Amer. Math. Soc.} 158 (1971), 301--308.



\bibitem{hk}
Hasselblatt, B., Katok, A.,
{\em Introduction to the modern theory of dynamical systems} (with a supplementary chapter by
Katok and Leonardo Mendoza),
Encyclopedia of Mathematics and its Applications, 54. Cambridge University Press, Cambridge, 1995.







\bibitem{hps}
Hirsch, M., Pugh, C., Shub, M.,
{\em Invariant manifolds},
Lec. Not. in Math. 583 (1977), Springer-Verlag.





\bibitem{k1}
Kuratowski, K.,
{\em Topology. Vol. II},
New edition, revised and augmented. Translated from the French by A. Kirkor Academic Press, New York-London;
Pa\'nstwowe Wydawnictwo Naukowe Polish Scientific Publishers, Warsaw 1968.


\bibitem{k}
Kuratowski, K.,
{\em Topology. Vol. I},
New edition, revised and augmented. Translated from the French by J. Jaworowski Academic Press, New York-London;
Pa\'nstwowe Wydawnictwo Naukowe, Warsaw 1966.







\bibitem{M1}
Ma\~n\'e, R.,
On the creation of homoclinic points,
{\em Inst. Hautes \'Etudes Sci. Publ. Math.} 66 (1988), 139--159. 



\bibitem{man}
Ma\~n\'e, R.,
{\em Ergodic theory and differentiable dynamics.
Translated from the Portuguese by Silvio Levy.}
Ergebnisse der Mathematik und ihrer Grenzgebiete (3) [Results in Mathematics and Related Areas (3)], 8. Springer-Verlag, Berlin, 1987. 

\bibitem{MM}
Ma\~n\'e, R.,
Oseledec's theorem from the generic viewpoint,
{\em Proceedings of the International Congress of Mathematicians}, Vol. 1, 2 (Warsaw, 1983), 1269--1276, PWN, Warsaw, 1984. 


\bibitem{M}
Ma{\~{n}}\'{e}, R.,
An ergodic closing lemma,
{\em Ann. of Math.} (2) 116 (1982), 503--540.


















\bibitem{m}
Morales, C.A.,
Another dichotomy for surface diffeomorphisms,
{\em Proc. Amer. Math. Soc.}  137  (2009),  no. 8, 2639--2644.





\bibitem{mp}
Morales, C.A., Pacifico, M.J.,
A dichotomy for three-dimensional vector fields,
{\em Ergodic Theory Dynam. Systems} 23 (2003), no. 5, 1575--1600. 








\bibitem{pt}
Palis, J., Takens, F.,
{\em Hyperbolicity and sensitive chaotic dynamics at homoclinic bifurcations.
Fractal dimensions and infinitely many attractors.} Cambridge Studies in Advanced Mathematics,
35. Cambridge University Press, Cambridge, 1993.








\bibitem{po}
Potrie, R.,
A proof of the existence of attractors,
Preprint (2009) unpublished.








\bibitem{PS}
Pujals, E.R., Sambarino, M.,
Homoclinic tangencies and hyperbolicity for surface diffeomorphisms,
{\em Ann. of Math.} (2) 151 (2000), 961--1023.






\bibitem{s}
Santiago, B.,
{\em Hiperbolicidade Essencial em Superf\'{i}cies} (Portuguese),
UFRJ/ IM, 2011.



\bibitem{sh}
Shub, M.,
{\em Global stability of dynamical systems}.
With the collaboration of Albert Fathi and R\'emi Langevin. Translated from the French by Joseph Christy.
Springer-Verlag, New York, 1987.




\bibitem{w}
Wen, L.,
Homoclinic tangencies and dominated splittings,
{\em Nonlinearity} 15 (2002), no. 5, 1445--1469.



\bibitem{w0}
Wen, L.,
On the preperiodic set,
{\em Discrete Contin. Dynam. Systems} 6 (2000), no. 1, 237--241.


\bibitem{w1}
Wen, L.,
On the $C^1$ stability conjecture for flows,
{\em J. Differential Equations} 129 (1996), no. 2, 334--357. 




\end{thebibliography}
\end{document}